\documentclass[12pt]{article}
\usepackage[width=16.5cm,height=21cm]{geometry}
\usepackage[colorlinks,allcolors=blue]{hyperref}
\usepackage{amsmath,amssymb,amsthm,cite}
\usepackage[boxsize=1.2ex,aligntableaux=center]{ytableau}
\usepackage{colortbl,tabularx}
\newcolumntype{K}[1]{>{\centering$\arraybackslash}m{#1}<{$}}

\numberwithin{equation}{section}
\newtheorem{thm}{Theorem}[section]
\newtheorem{lem}[thm]{Lemma}
\newtheorem{prop}[thm]{Proposition}
\newtheorem{cor}[thm]{Corollary}
\theoremstyle{definition}
\newtheorem{exmpl}[thm]{Example}
\newtheorem{rem}[thm]{Remark}

\newcommand{\matr}[2]{\left[\begin{array}{#1}#2\end{array}\right]}
\usepackage{colortbl}
\definecolor{deletedcolor}{rgb}{0.8,0.8,0.8}

\newcolumntype{d}{>{\columncolor{deletedcolor}[1.0\tabcolsep]}c}
\newcommand{\bC}{\mathbb{C}}
\newcommand{\bN}{\mathbb{N}}
\newcommand{\bZ}{\mathbb{Z}}
\newcommand{\adj}{\operatorname{adj}}
\newcommand{\al}{\alpha}
\newcommand{\be}{\beta}
\newcommand{\ga}{\gamma}
\newcommand{\la}{\lambda}
\newcommand{\si}{\sigma}
\newcommand{\id}{\operatorname{id}}
\newcommand{\reverse}[1]{\operatorname{rev}(#1)}
\newcommand{\total}[1]{|#1|}
\newcommand{\medstrut}{\vphantom{\int_{0_0}^{1^1}}}
\newcommand{\bigstrut}{\vphantom{\displaystyle\int_{0}^{1}}}
\newcommand{\detmatr}[2]{\left|\begin{array}{#1}#2\end{array}\right|}
\newcommand{\diag}{\operatorname{diag}}
\newcommand{\schv}{s^{\vphantom{n}}} 

\newcommand{\doi}[1]{\href{http://dx.doi.org/#1}{doi:#1}}
\newcommand{\myurl}[1]{\href{#1}{#1}}

\author{Egor A.\ Maximenko, Mario Alberto Moctezuma-Salazar}
\title{Cofactors and eigenvectors of banded Toeplitz matrices:\\
Trench formulas via skew Schur polynomials
\thanks{The research was partially supported
by IPN-SIP project, Instituto Polit\'{e}cnico Nacional, Mexico, and by Conacyt, Mexico.}}

\begin{document}
\maketitle

\begin{abstract}
The Jacobi--Trudi formulas imply that
the minors of the banded Toeplitz matrices
can be written as certain skew Schur polynomials.
In 2012, Alexandersson expressed the corresponding
skew partitions in terms of the indices of the struck-out rows and columns.
In the present paper, we develop the same idea and obtain some new applications.
First, we prove a slight generalization and modification of Alexandersson's formula.
Then, we deduce corollaries about the cofactors and eigenvectors of banded Toeplitz matrices,
and give new simple proofs to the corresponding formulas published by Trench in 1985.\\[1ex]
\end{abstract}

\tableofcontents

\section{Introduction}
\label{sec:intro}

The first exact formulas for banded Toeplitz determinants
were found by Widom \cite{Widom1958}
and by Baxter and Schmidt~\cite{BaxterSchmidt1961}.
Trench \cite{Trench1985inv,Trench1985eig} discovered another equivalent formula for the determinants and exact formulas for the inverse matrices and eigenvectors.
Among many recent investigations
on Toeplitz matrices and their generalizations we mention \cite{BG2017,BBGM2015,BBGM2016,BFGM2015,DeiftItsKrasovsky2013,
Elouafi2014,Elouafi2015,GaroniSerra2016}.
See also the books \cite{BG2005,BS1999,BS2006,GrenanderSzego1958}
which employ an analytic approach and contain
asymptotic results on Toeplitz determinants,
inverse matrices, eigenvalue distribution, etc.

It is obvious from the Jacobi--Trudi formulas
that there is a simple connection between Toeplitz minors and skew Schur polynomials.
Surprisingly for us, this connection was not mentioned in the works cited above.

Gessel \cite[Section~7]{Gessel1990} showed that some combinatorial generating functions can be written as Toeplitz determinants
(without actually naming them ``Toeplitz determinants'').
With the help of Gessel's formula, Borodin and Okounkov \cite{BorodinOkounkov2000}
expressed the general Toeplitz determinant as the Fredholm determinant of some operator acting on the Hilbert space $\ell^2$.
Tracy and Widom \cite{TracyWidom2001} used Gessel's formula
to compute some asymptotic distributions related to combinatorial objects.
They also observed that the minors located in the first columns of the triangular Toeplitz matrix $[h_{j-k}]_{j,k}$ can be written as Schur polynomials,
and expressed the corresponding partitions in terms of the selected rows.

Bump and Diaconis~\cite{BumpDiaconis2002}
considered Toeplitz minors with a fixed number of struck-out rows and columns, and studied their asymptotic behavior as the order of the minor tends to infinity.
They indexed the minors with two partitions without writing explicitly the relation between these partitions and the indices of the struck-out rows and columns.

Lascoux in his book~\cite{Lascoux2003}
defined skew Schur functions $s_{\la/\mu}$ as minors of the triangular Toeplitz matrix $[h_{k-j}]_{j,k}$ and explicitly related the partitions $\la$ and $\mu$
with the indices of the selected rows and columns.
Reiner, Shaw, and van Willigenburg mentioned the same relation
in their article \cite{ReinerShawWilligenburg2007}
about the problem of coinciding skew Schur polynomials.

Alexandersson \cite{A2012} found a new combinatorial proof
of Widom's formula for the determinants of Toeplitz matrices.
As an auxiliary result \cite[Proposition~3]{A2012}, he wrote the minors of triangular Toeplitz matrices $[e_{k-j}]_{j,k}$ as skew Schur polynomials $s_{\al/\be}$, with certain partitions $\al$ and $\be$ expressed explicitly in terms of the struck-out rows and columns.

The aim of this paper is to complement \cite{A2012}
by showing that some other classical results from the theory 
of Toeplitz matrices can also be naturally embedded
into the theory of skew Schur polynomials.
When it comes to building bridges between these two theories
we prefer to stay on the ``Toeplitz side.''
Thus, we start with a general (non-necessarily triangular)
banded Toeplitz matrix $T_n(a)$
generated by an arbitrary Laurent polynomial $a$
and express its minors as certain skew Schur polynomials
evaluated at the zeros of $a$.
Then we deduce several formulas for the cofactors and eigenvectors,
and give new proofs for the classical Trench's formulas.

\section{Main results}
\label{sec:results}

Let $a$ be a Laurent polynomial of the form
\begin{equation}\label{eq:Laurent}
a(t)=\sum_{k=p-w}^p a_k t^k
=a_p t^{p-w} \prod_{j=1}^{w} (t-z_j),
\end{equation}
where $p\in\bN_0=\{0,1,2,\ldots\}$,
$w\in\bN=\{1,2,\ldots\}$,
$a_{p-w},\ldots,a_p$ are some complex numbers
and $a_p\ne0$.
The coefficients $a_k$ are defined to be zero if $k>p$ or $k<p-w$.
For every $n$ in $\bN$ denote by $T_n(a)$ the $n\times n$ banded Toeplitz matrix generated by $a$:
\begin{equation}\label{eq:Toeplitz}
T_n(a)=[a_{j-k}]_{j,k=1}^n.
\end{equation}
Note that if $n$ is sufficiently large,
$p$ is the index of the last nonzero diagonal below the main diagonal
and $w+1$ is the width of the band.

By Vieta's formulas, the numbers $a_k/a_p$
can be written as elementary symmetric polynomials
in the variables $z_1,\ldots,z_w$, with alternating signs.
Thus, every minor of $T_n(a)$, after factorizing an appropriate power of $a_p$, 
is a symmetric polynomial in $z_1,\ldots,z_w$.
It turns out that it is a certain skew Schur polynomial, up to a sign.
Given a skew partition $\la/\mu$,
we denote by $s_{\la/\mu}(z_1,\ldots,z_w)$ or just by $s_{\la/\mu}$
the corresponding skew Schur polynomial in variables $z_1,\ldots,z_w$.
Given $k$ in $\bN_0$, we denote by $h_k(z_1,\ldots,z_w)$
the complete homogeneous polynomial of degree $k$.
The details are given in Section~\ref{sec:Schur}.

For every two integers $n,m$ such that $0\le m\le n$, we denote by $I^m_n$ the set of all strictly increasing functions $\{1,\ldots,m\}\to\{1,\ldots,n\}$.
Every element $\rho$ of $I^m_n$ can be written as an $m$-tuple $(\rho_1,\ldots,\rho_m)$, where
$\rho_1,\dots,\rho_m\in\{1,\ldots,n\}$ and
$\rho_1<\dots<\rho_m$.
We identify the function $\rho$ with the subset $\{\rho_1,\ldots,\rho_m\}$ of $\{1,\ldots,n\}$.
Let $\total{\rho}$ denote the sum $\rho_1+\cdots+\rho_m$.

Given $A$ in $\bC^{n\times n}$,
$m \in \{0,\ldots,n\}$, and $\rho,\si$ in $I^m_n$,
we denote by $A_{\rho,\si}$ the submatrix of $A$
located in the intersection of the rows
$\rho_1,\ldots,\rho_m$ and the columns $\si_1,\ldots,\si_m$:
\[
A_{\rho,\si}=
\bigl[A_{\rho_j,\si_k}\bigr]_{j,k=1}^m.
\]
Notice that if $m=0$, then the submatrix $A_{\rho,\si}$ is void,
and its determinant is $1$.

Working with integer tuples we use a comma to denote the concatenation
and superior indices to denote the repetition.
For example, $(5^3,3^2)=(5,5,5,3,3)$.
For every tuple $\xi=(\xi_1,\ldots,\xi_d)$
we denote by $\reverse{\xi}$ the reversed tuple
$(\xi_d,\ldots,\xi_1)$.
Let $\id_d$ denote the identity tuple
$(1,2,\ldots,d)$.

We start with two equivalent formulas for the minors of banded Toeplitz matrices.

\begin{thm}\label{thm:main}
Let $a$ be a Laurent polynomial of the form \eqref{eq:Laurent},
$n,m\in\bZ$, $0\le m\le n$,
$\rho,\si\in I^m_n$, $d=n-m$,
and $\xi,\eta\in I^d_n$ be the complements of $\rho,\si$, respectively.
Then
\begin{align}
\label{eq:main}
\det (T_n(a)_{\rho,\sigma})
&=(-1)^{pm+\total{\rho}+\total{\si}}\;a_p^m\;
\schv_{(m^p,\reverse{\xi-\id_d})/\reverse{\eta-\id_d}}(z_1,\ldots,z_w),
\intertext{and also}
\label{eq:Alexandersson}
\det (T_n(a)_{\rho,\sigma})
&=(-1)^{pm+\total{\rho}+\total{\si}}\;a_p^m\;
\schv_{(m^p,m^d+\id_d-\eta)/(m^d+\id_d-\xi)}(z_1,\ldots,z_w).
\end{align}
\end{thm}

Let us rewrite the skew partitions from
\eqref{eq:main} and \eqref{eq:Alexandersson} in the expanded form.
The skew Schur polynomial from
\eqref{eq:main} is $s_{\la/\mu}$, where
\begin{equation}\label{eq:lamu_through_xieta}
\lambda=(\underbrace{m,\ldots,m}_p,\xi_d-d,\ldots,\xi_1-1),\quad
\mu=(\eta_d-d,\ldots,\eta_1-1),
\end{equation}
and the skew Schur polynomial from
\eqref{eq:Alexandersson} can be written as $s_{\al/\be}$, where
\begin{equation}\label{eq:albe_through_xieta}
\al=(\underbrace{m,\ldots,m}_p,m+1-\eta_1,\ldots,m+d-\eta_d),\quad
\be=(m+1-\xi_1,\ldots,m+d-\xi_d).
\end{equation}
If a pair of partitions $\la,\mu$ does not form a skew partition, i.e.\ if $\la_j<\mu_j$ for some $j$, then we define the corresponding skew Schur polynomial $s_{\la/\mu}$ to be zero;
this convention is justified by Proposition~\ref{prop:not_skew_partition}.

Formula \eqref{eq:Alexandersson} is a simple generalization and modification of formula (3) from \cite{A2012}.
We prove Theorem~\ref{thm:main} in Section~\ref{sec:main}
combining ideas from \cite{A2012} with a couple of other tools.

In the particular case of Toeplitz determinants,
i.e.\ when $d=0$, both identities \eqref{eq:main} and \eqref{eq:Alexandersson} reduce to
\begin{equation}\label{eq:Tdet}
\det(T_n(a))
=(-1)^{pn} a_p^n \schv_{(n^p)}(z_1,\ldots,z_w).
\end{equation}
In fact, \eqref{eq:Tdet} follows immediately from the
dual Jacobi--Trudi formula for Schur polynomials,
without the need for skew Schur polynomials.
Formula \eqref{eq:Tdet} was noted by Bump and Diaconis 
\cite[proof of Theorem~1]{BumpDiaconis2002}.
Equivalent forms of \eqref{eq:Tdet} were discovered previously
by Baxter and Schmidt \cite{BaxterSchmidt1961}
and by Trench~\cite{Trench1985inv,Trench1985eig};
see Remark~\ref{rem:Baxter_Schmidt_and_Trench_formulas}.

Theorem~\ref{thm:main} says that every Toeplitz minor
can be written as a skew Schur polynomial.
Bump and Diaconis mentioned this fact in \cite[pp.~253 and 254]{BumpDiaconis2002},
but they did not express the skew partition in terms of the struck-out rows and columns.
Theorem~\ref{thm:main} also easily implies that every skew Schur polynomial in variables $z_1,\ldots,z_w$ can be written as an appropriate minor
of a Toeplitz matrix generated by a Laurent polynomial with zeros $z_1,\ldots,z_w$;
see Proposition~\ref{prop:skew_Schur_polynomial_as_Toeplitz_minor}.

Applying Theorem~\ref{thm:main},
we obtain the following equivalent formulas for the cofactors,
i.e.\ for the entries of the adjugate matrix.

\begin{thm}\label{thm:adj}
Let $a$ be of the form \eqref{eq:Laurent} and $n\ge1$.
Then for every $r,s$ in $\{1,\ldots,n\}$
\begin{align}
\label{eq:adj_through_skew_Schur}
\adj(T_n(a))_{r,s}
&=(-1)^{p(n-1)}\,a_p^{n-1}\,
\schv_{((n-1)^p,s-1)/(r-1)},
\\[1ex]
\label{eq:adj_through_skew_Schur_Alexandersson}
\adj(T_n(a))_{r,s}
&=(-1)^{p(n-1)}\,a_p^{n-1}\,\schv_{((n-1)^p,n-r)/(n-s)},
\\[1ex]
\label{eq:adj_through_Schur}
\adj(T_n(a))_{r,s}
&= (-1)^{p(n-1)}\,a_p^{n-1}\,
\sum_{k=\max(0,s-r)}^{\min(n-r,s-1)}
s_{((n-1)^{p-1},n+s-r-1-k,k)}\qquad(\text{for}\ p\ge1),
\\
\label{eq:Trench}
\adj(T_n(a))_{r,s}
&= (-1)^{pn}\,a_p^{n-1}
\left(h_{s-r-p} s_{(n^p)} - \sum_{k=0}^{p-1} (-1)^k h_{s+k-p} s_{(n^{p-k-1},(n-1)^k,n-r)}\right).
\end{align}
\end{thm}

The expansion in the right-hand side of \eqref{eq:adj_through_Schur}
does not make sense for $p=0$;
in this trivial case \eqref{eq:adj_through_skew_Schur},
\eqref{eq:adj_through_skew_Schur_Alexandersson},
and \eqref{eq:Trench} reduce to
$\adj(T_n(a))_{r,s}=a_0^{n-1} \schv_{(r-s)}=a_0^{n-1} h_{r-s}$.

Of course, \eqref{eq:adj_through_skew_Schur}--\eqref{eq:Trench} can be easily converted into formulas for the entries of the inverse matrix $T_n(a)^{-1}$:
suppose that $s_{(n^p)}\ne0$,
change the left-hand side for $(T_n(a)^{-1})_{r,s}$
and divide the right-hand sides by
$(-1)^{pn}a_p^n \schv_{(n^p)}$.

Formulas \eqref{eq:adj_through_skew_Schur}
and \eqref{eq:adj_through_skew_Schur_Alexandersson}
are immediate corollaries from Theorem~\ref{thm:main}.
In Section~\ref{sec:adj} we prove 
\eqref{eq:adj_through_Schur}
applying a particular case of the Littlewood--Richardson rule
and deduce \eqref{eq:Trench} from
\eqref{eq:adj_through_skew_Schur_Alexandersson}
using the Jacobi--Trudi formula
and expanding the corresponding determinant by the first column.
Note that \eqref{eq:Trench} is essentially Trench's formula
for the entries of the inverse matrix \cite[Theorem~3]{Trench1985inv}; 
it has an advantage over \eqref{eq:adj_through_Schur}
because the number of summands in \eqref{eq:Trench}
does not depend on $n$.

Finally, in the next theorem we
give a simple formula for the components of an eigenvector of $T_n(a)$, supposing that the associated eigenvalue $x$ is known.
More precisely, we construct a vector $v$ in $\bC^n$ satisfying $T_n(a)v=xv$;
there is no guarantee that $v$ is nonzero.

\begin{thm}\label{thm:eigvec}
Let $a$ be of the form \eqref{eq:Laurent},
$p\ge 1$, $n\ge 1$, and $x$ be an eigenvalue of $T_n(a)$.
Denote by $z_1,\ldots,z_w$ the zeros of $a-x$.
Then the vector $v=[v_r]_{r=1}^{n}$ with components
\begin{equation}\label{eq:eigvec_components}
v_r=s_{((n-1)^{p-1},n-r)}(z_1,\ldots,z_w)
\end{equation}
satisfies $T_n(a)v=xv$.
If $1\le p\le w$ and the zeros of $a-x$ are simple, then the vector $v$
can be written as a linear combination of geometric progressions
with ratios $1/z_1,\ldots,1/z_w$ and some complex coefficients
$C_1,\ldots,C_w$:
\begin{equation}\label{eq:eigvec_Trench}
v_r = \sum_{j=1}^w C_j^{\vphantom{n-r+w-p}} z_j^{n-r+w-p}.
\end{equation}
\end{thm}

In Section~\ref{sec:eigvec} we prove Theorem~\ref{thm:eigvec}
and provide explicit expressions for $C_1,\ldots,C_w$.
Formula \eqref{eq:eigvec_Trench} essentially coincides
with Trench's formula \cite[Theorem~1]{Trench1985eig}.

Theorems~\ref{thm:main}, \ref{thm:adj}, \ref{thm:eigvec},
except for \eqref{eq:eigvec_Trench}, are also true for Toeplitz matrices generated by semi-infinite Laurent series; in this situation one has to work with skew Schur functions instead of skew Schur polynomials; see Remark~\ref{rem:Laurent_series}.

All formulas from Theorems~\ref{thm:main}--\ref{thm:eigvec},
except for \eqref{eq:adj_through_Schur},
are very efficient for small values of $w$, $p$, and $d$, even if $n$ is large.

We tested the main results with symbolic and numerical computations;
see the details in Remark~\ref{rem:test}.

\clearpage
\section{Skew Schur polynomials: notation and facts}
\label{sec:Schur}

In this section, we fix some notation
and recall some well-known facts about skew Schur polynomials.
See \cite{Macdonald1995} and \cite{Stanley1999}
for explanations and proofs.

For each $r$ in $\bN_0$,
the $r$-th \emph{elementary symmetric polynomial} $e_r$
is the sum of all products of $r$ distinct variables,
and the $r$-th \emph{complete homogeneous symmetric polynomial} $h_r$ 
is the sum of all monomials of total degree $r$:
\[
e_r(x_1,\ldots,x_w) = \sum_{j_1<j_2<\dots<j_r}\prod_{k=1}^r x_{j_k},\qquad
h_r(x_1,\ldots,x_w) = \sum_{j_1,\ldots,j_r\in\{1,\ldots,w\}} \prod_{k=1}^w x_{j_k}.
\]
When $x_1,\ldots,x_w$ are pairwise different complex numbers,
$h_r(x_1,\ldots,x_w)$ can be computed by the following efficient formula:
\begin{equation}\label{eq:h_through_x}
h_r(x_1,\ldots,x_w)
=\sum_{j=1}^w\frac{x_j^{r+w-1}}{\displaystyle\prod_{k\in\{1,\ldots,w\}\setminus\{j\}}(x_j-x_k)}.
\end{equation}
For brevity, from now on we omit the variables $x_1,\ldots,x_w$ when possible.
Note that $h_0=e_0=1$ and $h_1=e_1$.
Define $h_r$ and $e_r$ to be zero for $r<0$.
The above definitions also imply that $e_r=0$ for $r>w$.
Consider the generating functions for the sequences of polynomials $(e_r)_{r=0}^\infty$ and $(h_r)_{r=0}^\infty$:
\[
E(t) = \sum_{r=0}^w e_r t^r = \prod_{s=1}^w (1+x_s t),\qquad
H(t) = \sum_{r\ge0} h_r t^r = \prod_{s\ge0} (1-x_s t)^{-1}.
\]
The formal series $E(-t)$ and $H(t)$ are mutually reciprocal:
\begin{equation}\label{eq:prod_gen_func}
E(-t) H(t) = 1.
\end{equation}
Equivalently, the sequences $(e_r)_{r=0}^\infty$
and $(h_r)_{r=0}^\infty$ are related by
\begin{equation}\label{eq:e_h_relation}
\sum_{k=0}^j (-1)^k e_k h_{j-k} = \delta_{j,0}\qquad(j\in\bN_0).
\end{equation}

\emph{Integer partitions} (\emph{partitions}, for short)
are tuples of the form
$\la=(\la_1,\la_2,\ldots,\la_m)$,
where $m\in\bN_0$,
$\la_1,\la_2,\ldots,\la_m\in\bZ$,
and $\la_1\ge\la_2\ge\ldots\ge\la_m\ge0$.
Each $\la_j$ is called a \emph{part} of $\la$.
The number of parts is the \emph{length} of $\la$, denoted by $\ell(\la)$;
and the sum of the parts is the \emph{weight} of $\la$, denoted by $|\la|$.
The \emph{Young--Ferrers diagram} of a partition $\la$ may be
formally defined as the set of all points $(j,k)$ in $\bZ^2$ such that $1\le k \le \la_j$.
The \emph{conjugate} of a partition $\lambda$ is the partition
$\lambda'$ whose diagram is the transpose of the diagram $\lambda$.
Hence, $\lambda'_k$ is the number of nodes in the $k$th column of $\lambda$:
\begin{equation}\label{la_conj}
\lambda'_k = \#\{j\colon \lambda_j\ge k\}.
\end{equation}
Here is an example of a partition and its conjugate,
shown as Young--Ferrers diagrams:
\[
\la=(5,3)=\ydiagram{5,3}\;,\qquad
\la'=(2,2,2,1,1)=\ydiagram{2,2,2,1,1}\;.
\]
Given two partitions $\la$ and $\mu$,
it will be written $\mu\subseteq\la$ if $\ell(\mu)\le\ell(\la)$
and $\mu_j\le\la_j$ for all $j$ in $\{1,\ldots,\ell(\mu)\}$.
In this case, the pair $\la,\mu$ is called a \emph{skew partition} and is denoted by $\la/\mu$.
The diagram of $\la/\mu$ is defined as the set difference of the diagrams associated to $\la$ and $\mu$.
For example,
\[
(7,4,2)/(2,1)=\ydiagram{2+5,1+3,2}\;.
\]
There are many equivalent definitions of skew Schur polynomials.
Given a skew partition $\la/\mu$,
the \emph{skew Schur polynomial} $s_{\lambda/\mu}$
can be defined by the (first) \emph{Jacobi--Trudi formula}:
\begin{equation}\label{SS_JT1}
s_{\la/\mu} = \det[h_{\lambda_j-\mu_k-j+k}]_{j,k=1}^{\ell(\la)}.
\end{equation}
The \emph{dual Jacobi--Trudi formula}
(also known as the \emph{second Jacobi--Trudi formula}
or the \emph{N\"{a}gelsbach--Kostka formula})
expresses $s_{\la/\mu}$ in terms of the elementary polynomials:
\begin{equation}\label{SS_JT2}
s_{\la/\mu} = \det[e_{\lambda_j'-\mu_k'-j+k}]_{j,k=1}^{\la_1}.
\end{equation}
In formulas \eqref{SS_JT1} and \eqref{SS_JT2},
the partition $\mu$ is extended with zeros
up to the length of the partition $\la$.
Also notice that the extension of $\la$
with zeros does not change
the corresponding skew Schur polynomial $s_{\la/\mu}$,
though the matrix $JT(\la/\mu)$ changes its size.

We use the notation $\la/\mu$ and the definition
\eqref{SS_JT1} for all pairs of partitions $\la,\mu$,
without requiring $\mu\subseteq\la$,
and denote by $JT(\la/\mu)$ the matrix
$[h_{\lambda_j-\mu_k-j+k}]_{j,k=1}^{\ell(\la)}$
that appears in the right-hand side of the formula \eqref{SS_JT1}.

\begin{prop}\label{prop:not_skew_partition}
For every integer partitions $\la$ and $\mu$ such that $\mu\nsubseteq\la$, 
the determinants in the right-hand sides of \eqref{SS_JT1}
and \eqref{SS_JT2} are zero.
\end{prop}

\begin{proof}
Suppose that $p\in\{1,\ldots,\ell(\la)\}$ and $\mu_p>\la_p$.
Then the $(p,p)$th entry of $JT(\la/\mu)$ is zero,
and so are the entries below and to the left of $(p,p)$.
Indeed, if $j\ge p$ and $k\le p$, then
\[
\la_j-\mu_k-j+k \le \la_p - \mu_p < 0
\]
and therefore $JT(\la/\mu)_{j,k}=0$.
Thus, the rank of the matrix formed by the first $p$ columns
of $JT(\la/\mu)$ is strictly less than $p$,
the matrix $JT(\la/\mu)$ is singular,
and $\det(JT(\la/\mu))=0$.
Another way to obtain the same conclusion
is to divide the rows and columns of $JT(\la/\mu)$
into the parts $\{1,\dots,p\}$ and $\{p+1,\dots,w\}$
and to apply the formula for the determinant of block-triangular matrices.

The situation with the determinant in \eqref{SS_JT2}
is similar; let us just prove that
$\mu\subseteq\la$ if and only if $\mu'\subseteq\la'$.
Since the operation $\la\mapsto\la'$ is involutive,
it is sufficient to verify the \emph{if} part.
Suppose that $\mu_k\le\la_k$ for each $k$.
Then for each $j$ we obtain
$\{k\colon\ \mu_k\ge j\} \subseteq \{k\colon\ \la_k\ge j\}$,
thus, $\mu_j'\le\la_j'$ by \eqref{la_conj}.
\end{proof}

\medskip
In particular, if $\mu$ is the void partition $()$,
then the skew Schur polynomial $s_{\la/\mu}$ is called the \emph{Schur polynomial}
associated to $\la$ and is denoted by $s_\la$.
In this case, the Jacobi--Trudi formula~\eqref{SS_JT1} 
and its dual \eqref{SS_JT2} become
\begin{align}
\label{eq:JT1_Schur}
s_\la &= \det [h_{\la_j-j+k}]_{j,k=1}^{\ell(\la)}, \\
\label{eq:JT2_Schur}
s_\la &= \det [e_{\la_j'-j+k}]_{j,k=1}^{\la_1}.
\end{align}

\begin{rem}\label{rule:partition_subindex_h}
The parts of $\la$ are the degrees of the complete homogeneous polynomials in the main diagonal of $JT(\la)$, i.e. $JT(\la)_{j,j}=h_{\la_j}$.
\end{rem}

Schur polynomials can also be expressed as quotients of antisymmetric functions.
Notice first that if $\ell(\la)>w$, then \eqref{eq:JT2_Schur} yields
$s_\la(x_1,\ldots,x_w)=0$.
If $\ell(\la)\le w$, then the partition $\la$
can be extended with zeros up to the length $w$, and
\begin{equation}\label{eq:Schur_quotient}
s_\la(x_1,\ldots,x_w)
=\frac{\det A_\la(x_1,\ldots,x_w)}{\det A_{(0^w)}(x_1,\ldots,x_w)},
\end{equation}
where $A_\la(x_1,\ldots,x_w)$ is the generalized Vandermonde matrix
\begin{equation}\label{eq:generalized_Vandermonde}
A_\la(x_1,\ldots,x_w)
=\bigl[ x_k^{\la_j+w-j} \bigr]_{j,k=1}^w
=\matr{cccc}{
\medstrut x_1^{\la_1+w-1} & x_2^{\la_1+w-1} & \cdots & x_w^{\la_1+w-1}\\
\medstrut x_1^{\la_2+w-2} & x_2^{\la_2+w-2} & \cdots & x_w^{\la_2+w-2}\\
\medstrut \vdots & \vdots & \ddots & \vdots \\
\medstrut x_1^{\la_w} & x_2^{\la_w} & \cdots & x_w^{\la_w}
}.
\end{equation}
The denominator of the quotient in \eqref{eq:Schur_quotient}
is the Vandermonde determinant
\[
V(x_1,\ldots,x_w)
=\det A_{(0^w)}(x_1,\ldots,x_w)
=\prod_{1\leq j<k\leq w}(x_j-x_k).
\]

\begin{rem}\label{rem:confluent}
If some of the numbers $x_1,\ldots,x_w$ coincide,
then the quotient \eqref{eq:Schur_quotient} can be computed by the 
L'H\^{o}pital's rule, i.e.\ by differentiating the columns of the determinants
that correspond to the repeating variables.
Following Trench \cite[Definition~1]{Trench1985inv}
and modifying slightly his notation,
we denote by $C_\la$ the vector-function $t\mapsto [t^{\la_j+w-j}]_{j=1}^w$,
and by $C_\la^{(q)}$ the $q$-th derivative of this function.
Suppose that the list $x_1,\ldots,x_w$ contains $\ga$ different
complex numbers, and $m_1,\ldots,m_\ga$ are their multiplicities,
such that $m_1+\dots+m_\ga=w$,
$z_1=\dots=z_{m_1}$, $z_{m_1+1}=\dots=z_{m_1+m_2}$, and so forth.
Denote by $A_\la(x_1,\ldots,x_w)$
the \emph{confluent generalized Vandermonde matrix} of order $w$,
whose first $m_1$ columns are $C^{(0)}(x_1),\dots,C^{(m_1-1)}(x_1)$,
the next $m_2$ columns are $C^{(0)}(x_{m_1+1}),\dots,C^{(m_2-1)}(x_{m_1+1})$, etc.
Then $V=\det A_{(0^w)}(x_1,\ldots,x_w)$ is the confluent Vandermonde determinant.
With these modifications, formula \eqref{eq:Schur_quotient} is valid
in the case when some of $x_1,\ldots,x_w$ coincide.
\end{rem}

Schur polynomials form a basis for the vector space
of homogeneous symmetric polynomials.
In particular, the skew Schur polynomials can be expressed
as sums of Schur polynomials. 
This is known as the Littlewood--Richardson rule.

In this paper we only use a particular case of the Littlewood--Richardson rule,
known as the \emph{skew version of Pieri's rule}.
Let $\la$ be a partition and $r\in\bN$.
Then
\begin{equation}\label{eq:skew_Pieri}
s_{\la/(r)} = \sum_\nu s_\nu,
\end{equation}
where $\nu$ ranges over all partitions $\nu\subseteq\la$
such that $\la/\nu$ is a horizontal strip of size $r$.
It is said that a skew partition is a \emph{horizontal strip}
if in the associated diagram there is no two nodes in the same column.
Formally, $\nu$ from \eqref{eq:skew_Pieri} satisfies
\begin{equation}\label{eq:formal_skew_Pieri}
\ell(\nu)\le\ell(\la),\quad
\la_{j+1}\le\nu_j\le\la_j\ (1\le j\le\ell(\la))
\quad\text{and}\quad
|\la|-|\nu|=r.
\end{equation}

The \emph{flip operation} over skew partitions is defined by
\begin{equation}\label{eq:flip}
(\la/\mu)^\ast = (\la_1^{\ell(\la)}-\reverse{\mu})/(\la_1^{\ell(\la)}-\reverse{\la}),
\end{equation}
i.e. $(\la/\mu)^\ast=\al/\be$, where
\begin{equation}\label{eq:flip_al_bet}
\al_j = \la_1-\mu_{{\ell(\la)}+1-j}, \qquad \be_j = \la_1-\la_{{\ell(\la)}+1-j}.
\end{equation}
Here is an example of the action of the flip operation:
\begin{align*}
\la/\mu&=(6,6,3,1)/(3,2,2)=\ydiagram{3+3,2+4,2+1,1}\;,\\
(\la/\mu)^\ast&=\al/\beta=(6,4,4,3)/(5,3)=\ydiagram{5+1,3+1,4,3}\;.
\end{align*}

It is well known that $s_{(\la/\mu)^\ast}=s_{\la/\mu}$;
see, for example, \cite[Exercise, 7.56(a)]{Stanley1999}.
For the reader's convenience, we give another simple proof of this fact
in the following proposition by applying the Jacobi--Trudi formula
and the concept of the pertranspose matrix.
Recall that the \emph{exchange matrix} of order $n$ is defined by
$J_n=[\delta_{j+k,n+1}]_{j,k=1}^n$,
where $\delta$ is Kronecker's delta.
Given a matrix $A \in \bC^{n\times n}$,
the matrix $J_n A^\top J_n$ is called the \emph{pertranspose} of $A$;
its entry $(j,k)$ equals $A_{n+1-k,n+1-j}$.
The matrices $J_n A^\top J_n$ and $A$
have the same determinant because
$\det(J_n)=(-1)^{n(n-1)/2}$.


\begin{prop}\label{prop:skew_Schur_flip}
Let $\la$ and $\mu$ be some partitions. Then
\begin{equation}\label{eq:skew_Schur_flip}
s_{\la/\mu} = s_{(\la/\mu)^\ast}.
\end{equation}
\end{prop}

\begin{proof}
Denote the length of $\la$ by $n$.
The following computation shows that
the matrix $JT((\la/\mu)^\ast)$ is the pertranspose of $JT(\la/\mu)$:
\begin{align*}
JT((\la/\mu)^\ast)_{j,k}
&=h_{\la_{n+1-k}-\mu_{n+1-j}-j+k}
\\
&=h_{\la_{n+1-k}-\mu_{n+1-j}-(n+1-k)+(n+1-j)}
=JT(\la/\mu)_{n+1-k,n+1-j}.
\end{align*}
Thus, the determinants of $JT((\la/\mu)^\ast)$ and $JT(\la/\mu)$ coincide.
\end{proof}

If the pair $\la,\mu$ does not form a skew partition,
i.e.\ $\la_j<\mu_j$ for some $j$, then the pair $\al,\be$
defined by \eqref{eq:flip_al_bet} have the same defect,
hence in this case both sides of \eqref{eq:skew_Schur_flip} are zero.

\begin{rem}\label{rem:skew_Schur_functions}
Many formulas for skew Schur polynomials
do not involve explicitly the variables $x_1,\ldots,x_w$;
the corresponding abstraction leads to the concept of \emph{skew Schur functions}.
In this paper, it is convenient to think that the argument of a skew Schur function is not a list of variables $x_1,x_2,\ldots$,
but rather an arbitrary sequence of numbers $(e_k)_{k=0}^\infty$ or,
equivalently, the formal series $E(t)=\sum_{k=0}^\infty e_k t^k$. 
Then the numbers $h_0,h_1,\ldots$ are defined as the coefficients of the formal series $H(t)=1/E(-t)$, 
i.e.\ by formula \eqref{eq:e_h_relation},
and the skew Schur functions
are defined by Jacobi--Trudi formula \eqref{SS_JT1}.
This idea is explained in \cite[pp.~99--100 and Chap.~VII]{Littlewood1950}
and \cite[Exercise~7.91]{Stanley1999}.
Note that \eqref{eq:Schur_quotient} makes sense only for Schur polynomials.
\end{rem}

\section{Minors of banded Toeplitz matrices}
\label{sec:main}

In this section we prove Theorem~\ref{thm:main}.
Let $a$ be a Laurent polynomial of the form \eqref{eq:Laurent}.
Consider the polynomial
\[
P(t)=t^{w-p} a(t)
=\sum_{j=0}^w a_{j+w-p} t^j
=a_p \prod_{j=1}^w (t-z_j).
\]
By Vieta's formulas, the coefficients of this polynomial
can be related with elementary symmetric polynomials
$a_{j+p-w}=a_p (-1)^{w-j} e_{w-j} t^j$, which yields
\begin{equation}\label{eq:a_through_e}
a_j=(-1)^{p-j}a_p e_{p-j}.
\end{equation}
Note that \eqref{eq:a_through_e} is true not only for $p-w\le j\le p$,
but for every integer $j$, 
since for $j>p$ or $j<p-w$ both sides equal zero.
Thus the matrix $T_n(a)$ can be written as
\begin{equation}\label{eq:Tmatrix_through_e}
T_n(a)=(-1)^p\,a_p\,[(-1)^{j+k} e_{p+k-j}]_{j,k=1}^n.
\end{equation}

\begin{rem}[Toeplitz matrices generated by semi-infinite Laurent series]
\label{rem:Laurent_series}
Instead of working with Laurent polynomials of the form
\eqref{eq:Laurent} and corresponding banded Toeplitz matrices,
we consider the more general situation in which
$a$ is a formal Laurent series of the form
$a(t)=\sum_{k=-\infty}^p a_k t^k$,
with $a_p\ne0$.
In other words, the initial data is the sequence
$(a_k)_{k\in\bZ}$ such that $a_p\ne0$ and $a_k=0$ for $k>p$.
Then the corresponding Toeplitz matrices have only a finite number
of nonzero diagonals below the leading diagonal.
We do not know how to define the zeros of a formal Laurent series
nor how to express the coefficients $a_k$ through these zeros.
Thus we define $(e_k)_{k=0}^\infty$ directly by \eqref{eq:a_through_e}
and work with skew Schur functions, see Remark~\ref{rem:skew_Schur_functions}.
Note that $a$ can be written in terms of $(e_k)_{k=0}^\infty$ as
\begin{equation}\label{eq:aseries_through_e}
a(t)
=\sum_{k=-\infty}^p a_k t^k
=a_p t^p \sum_{k=-\infty}^0 (-1)^k e_{-k}t^k
=a_p t^p E(-1/t).
\end{equation}
Almost all results of this paper are valid in this situation,
except for \eqref{eq:eigvec_Trench}.
\end{rem}


Consider first the triangular case studied by other authors,
in order to facilitate the comparison of the results.
Let
\begin{equation}\label{eq:bseries_through_e}
b(t)=E(1/t)=\sum_{k=0}^\infty e_k t^{-k}
=\sum_{k=-\infty}^0 e_{-k} t^k.
\end{equation}
If $e_0,e_1,\ldots$ are elementary symmetric polynomials in a finite number of variables,
$z_1,\ldots,z_w$, then the sums in \eqref{eq:bseries_through_e} are finite,
and $-z_1,\ldots,-z_w$ are the roots of $b$.

The next lemma is a simple application of the dual Jacobi--Trudi formula.

\begin{lem}\label{lem:triangular_minor_by_JT1}
Let $n\ge 1$, $m\le n$, $\rho,\si\in I^m_n$, and $d=n-m$.
Then
\begin{equation}\label{eq:minor_triangular_rhosi}
\det(T_n(b)_{\rho,\si})
=\det \bigl[ e_{\si_k-\rho_j} \bigr]_{j,k=1}^m
=s_{(\id_m-\rho+d^m)'/(\id_m-\si+d^m)'}.
\end{equation}
\end{lem}

\begin{proof}
The $(j,k)$th entry of the matrix
$T_n(b)_{\rho,\si}$ is $e_{\si_k-\rho_j}$.
Rewrite the difference $\si_k-\rho_j$ as
\[
\si_k-\rho_j
=(j-\rho_j+d)-(k-\si_k+d)-j+k,
\]
and define partitions $\la$ and $\mu$ via their conjugates:
\begin{equation}\label{eq:lamuprime_triangular_expanded}
\la'_j=j-\rho_j+d,\qquad
\mu'_k=k-\si_k+d.
\end{equation}
Then $T_n(b)_{\rho,\si}$ coincides with the matrix
from the right-hand side of \eqref{SS_JT2}.
\end{proof}

The summand $d$ in \eqref{eq:lamuprime_triangular_expanded}
serves just to make $\la'_j$ and $\mu'_k$ non-negative. 
This summand can be substituted by any integer greater
or equal to $\max\{\rho_m-m,\si_m-m\}$.

Lemma~\ref{lem:triangular_minor_by_JT1}
is very close to the Subsection~4.2 from \cite{ReinerShawWilligenburg2007},
but the Toeplitz matrix considered there is
$[h_{k-j}]_{j,k=0}^\infty$,
and the summand $\si_m-m$ is used instead of $d$.
See also the Section~1.4 from \cite{Lascoux2003},
where the skew Schur functions are defined as certain minors of the matrix
$[h_{k-j}]_{j,k=0}^\infty$, and the parts of the partitions
are written in the increasing order.

If the number of the selected rows and columns
(which we denote by $m$) is small,
then the corresponding Toeplitz minor is easy to compute directly.
The interesting case is when $d=n-m$ is small,
i.e.\ when we strike out only few rows and columns.
We are going to express $\la$ and $\mu$
in terms of the deleted rows and columns.
The following elementary combinatorial lemma yields the
ascending enumeration of a subset of $\{1,\ldots,n\}$
in terms of its complement.

\begin{lem}\label{lem:complement}
Let $m,n\in\bN$, $m\le n$, $d=n-m$,
$\rho\in I^m_n$ and $\xi\in I^d_n$ such that $\xi=\{1,\ldots,n\}\setminus \rho$.
Then for every $j$ in $\{1,\dots,d\}$
\begin{equation}\label{eq:complement}
\xi_j = j + \#\{k\in\{1,\dots,m\}\colon\ \rho_k-k < j\}.
\end{equation}
\end{lem}

\begin{proof}
Put $A_j=\{k\in\{1,\dots,m\}\colon\ \rho_k<\xi_j\}$
and $B_k=\{j\in\{1,\dots,d\}\colon\ \xi_j<\rho_k\}$.
The set $A_j$ is the complement of $\{\xi_1,\ldots,\xi_j\}$
in $\{1,\ldots,\xi_j\}$, therefore $\xi_j=j+\#A_j$.
Similarly, $\rho_k=k+\#B_k$.
Furthermore,
\[
\rho_k<\xi_j
\quad\Longleftrightarrow\quad
j\notin B_k
\quad\Longleftrightarrow\quad
\#B_k<j
\quad\Longleftrightarrow\quad
\rho_k-k<j,
\]
and $\xi_j=j+\#A=j+\#\{k\colon\ \rho_k-k<j\}$.
\end{proof}

To illustrate Lemma~\ref{lem:complement},
take $n=10$ and $\rho=(1,3,4,6,9,10)$, then $\xi=(2,5,7,8)$ and
\[
\xi_3
=3+\#\{k\colon\ \rho_k-k<3\}
=3+\#\{1,2,3,4\}=7.
\]

By comparing formulas \eqref{eq:complement} and \eqref{la_conj}
we see that the duality between the sets and their complements
is somehow similar to the duality between the partitions and their conjugates.
This is the key idea behind the proof of the following lemma.


\begin{lem}\label{lem:triangular_minor_xi_eta}
In conditions of Lemma~\ref{lem:triangular_minor_by_JT1},
denote by $\xi$ and $\eta$ the complements of $\rho$ and $\si$,
with respect to the set $\{1,\ldots,n\}$.
Then
\begin{equation}\label{eq:minor_triangular_xieta}
\det(T(b)_{\rho,\si})
=s_{\reverse{\xi-\id_d}/\reverse{\eta-\id_d}}
\end{equation}
and
\begin{equation}\label{eq:Alexandersson_triangular_xieta}
\det(T(b)_{\rho,\si})
=s_{(m^d+\id_d-\eta)/(m^d+\id_d-\xi)}.
\end{equation}
\end{lem}

\begin{proof}
Let us compute $\la_j$ by applying \eqref{la_conj},
\eqref{eq:lamuprime_triangular_expanded} and \eqref{eq:complement}.
For every $j\in\{1,\ldots,d\}$,
\begin{align*}
\la_j
&= \#\{k\colon\ \la_k'\geq j\}
= \#\{k\colon\ k - \rho_k + d \geq j\}
\\[0.5ex]
&= \#\{k\colon\ \rho_k - k \le d - j\}
= \#\{k\colon\ \rho_k - k < d + 1 - j\}
\\[0.5ex]
&= \xi_{d+1-j}-(d+1-j)
= (\xi-\id_d)_{d+1-j}
= \reverse{\xi-\id_d}_j.
\end{align*}
Thereby we obtain $\la=\reverse{\xi-\id_d}$.
The formula $\mu=\reverse{\eta-\id_d}$ can be proved in a similar manner.
As we know from Proposition~\ref{prop:skew_Schur_flip},
$s_{\la/\mu}=s_{\al/\be}$, where $\al/\be=(\la/\mu)^\ast$.
We compute $\al$ and $\be$ by \eqref{eq:flip}:
\[
\al=\la_1^{\ell(\la)} - \reverse{\mu}
=\la_1^{\ell(\la)} + \id_d - \eta,\qquad
\be=\la_1^{\ell(\la)} - \reverse{\la}
=\la_1^{\ell(\la)} + \id_d - \xi.
\]
Note that $\la_1=\xi_d-d\le n-d=m$.
The partitions $m^d+\id_d-\eta$ and $m^d+\id_d-\xi$
are obtained from $\al$ and $\be$
by adding the same non-negative integer number
$m-\la_1$ to all parts.
This operation does not change the skew Schur function,
thus \eqref{eq:Alexandersson_triangular_xieta} is proved.
\end{proof}

Formula~\eqref{eq:Alexandersson_triangular_xieta}
is the original Alexandersson's formula (3) from \cite{A2012}.
The statements and proofs of
our Lemmas~\ref{lem:triangular_minor_by_JT1}--\ref{lem:triangular_minor_xi_eta}
may be viewed as another redaction of the proof
of Proposition~3 in \cite{A2012}.

In what follows we use extensively the obvious identity
\begin{equation}\label{eq:id_sum}
\id_{d+p}=(\id_p,p^d+\id_d)=(\id_d,d^p+\id_p).
\end{equation}

The next lemma shows how to transform a general minor of $T_n(a)$
to a certain minor of the triangular matrix $T_{n+p}(b)$.
As usual, in the notation for the submatrices
we have to indicate the selected rows and columns,
but the real ``characters of the tale'' are the struck-out rows and columns.

\begin{lem}\label{lem:from_minor_Ta_to_minor_Tb}
Let $a$ and $b$ be Laurent series of the forms~\eqref{eq:aseries_through_e}
and \eqref{eq:bseries_through_e}, respectively.
Furthermore, let $n,m\in\bZ$, $1\le m\le n$, $\rho,\si\in I^m_n$,
$\xi=\{1,\ldots,n\}\setminus\rho$, $\eta=\{1,\ldots,n\}\setminus\si$.
Then
\begin{equation}\label{eq:minor_Ta_to_Tb}
\det T_n(a)_{\rho,\si}
=(-1)^{|\rho|+|\si|+mp} a_p^m
\det T_{n+p}(b)_{\widehat{\rho},\widehat{\si}},
\end{equation}
where 
$\widehat{\xi}=(\xi,\id_p+n^p)$,
$\widehat{\eta}=(\id_p,\eta+p^d)$,
$\widehat{\rho}=\{1,\dots,n+p\}\setminus\widehat{\xi}$,
and
$\widehat{\si}=\{1,\ldots,n+p\}\setminus\widehat{\eta}$.
\end{lem}

\begin{proof}
First we express the submatrix $T_n(a)_{\rho,\si}$
in terms of $(e_k)_{k=0}^\infty$:
\[
T_n(a)_{\rho,\si}
=[a_{\rho_j-\si_k}]_{j,k=1}^m
=(-1)^p\,a_p\,[(-1)^{\rho_j+\si_k} e_{p+\si_k-\rho_j}]_{j,k=1}^m.
\]
Now we factorize $(-1)^{\rho_j}$ from the $j$-st row, for every $j$,
and $(-1)^{\si_k}$ from the $k$-st column, for every $k$.
Formally,
\[
T_n(a)_{\rho,\si}
=(-1)^p\,a_p\;
\diag([(-1)^{\rho_j}]_{j=1}^m)\;
[e_{p+\si_k-\rho_j}]_{j,k=1}^m\;
\diag([(-1)^{\si_k}]_{k=1}^m),
\]
where $\diag(v)$ is the diagonal matrix generated by a vector $v$.
Passing to determinants we obtain
\begin{equation}\label{eq:Tminor_through_dete}
\det T_n(a)_{\rho,\si}
=(-1)^{\total{\rho}+\total{\si}+mp}\,a_p^m\,
\det [ e_{p+\si_k-\rho_j}]_{j,k=1}^m.
\end{equation}
Defining $\widehat{\rho},\widehat{\si}\in I^m_{n+p}$ by
$\widehat{\rho}=\rho$ and $\widehat{\si}=\si+p^m$
we arrive at \eqref{eq:minor_Ta_to_Tb}.
We are left to compute the complements of $\widehat{\rho}$
and $\widehat{\si}$ in $\{1,\ldots,n+p\}$.
Identifying subsets with strictly increasing tuples
and using \eqref{eq:id_sum} we obtain
\begin{align*}
\id_{n+p}\setminus\widehat{\rho}
&=(\id_n \setminus \rho)\cup (\id_p+n^p)
=(\xi,\id_p+n^p)
=\widehat{\xi},
\\
\id_{n+p}\setminus\widehat{\si}
&=\id_p \cup\left((\id_n + p^n) \setminus (\si+p^m)\right)
=\id_p \cup (\eta+p^d)=(\id_p,\xi+p^d)=\widehat{\eta}.
\qedhere
\end{align*}
\end{proof}

\begin{proof}[Proof of Theorem~\ref{thm:main}]
By Lemma~\ref{lem:from_minor_Ta_to_minor_Tb}
and Lemma~\ref{lem:triangular_minor_xi_eta}
applied to $\widehat{\xi}$ and $\widehat{\eta}$,
\[
\det T_n(a)_{\rho,\si}
=(-1)^{|\rho|+|\si|+pm} a_p^m s_{\la/\mu}
=(-1)^{|\rho|+|\si|+pm} a_p^m s_{\al/\be},
\]
where
\begin{align*}
\la &= \reverse{\widehat{\xi}-\id_{d+p}}, &
\mu &= \reverse{\widehat{\eta}-\id_{d+p}},
\\
\al &= m^{d+p}+\id_{d+p}-\widehat{\eta}, &
\be &= m^{d+p}+\id_{d+p}-\widehat{\xi}.
\end{align*}
With the help of \eqref{eq:id_sum},
these expressions are easily transformed
to the partitions
\eqref{eq:lamu_through_xieta} and \eqref{eq:albe_through_xieta}
from Theorem~\ref{thm:main}.
For example,
\[
\widehat{\xi}-\id_{d+p}
=(\xi,\id_p+n^p)-(\id_d,\id_p+d^p)
=(\xi-\id_d,m^p),
\]
thus $\la=\reverse{(\xi-\id_d,m^p)}=(m^p,\reverse{\xi-\id_d})$.
\end{proof}

\begin{rem}\label{rem:another_way}
Another way to prove Theorem~\ref{thm:main}
is to apply formula \eqref{SS_JT2} and Lemma~\ref{lem:complement}
directly to $\det (T_n(a))_{\rho,\si}$,
without reducing the situation to the triangular case.
The equality between the right-hand sides of
\eqref{eq:main} and \eqref{eq:Alexandersson}
can also be proved directly,
by verifying that $\al/\be=(\la/\mu)^\ast$
or by applying the persymmetric property of Toeplitz matrices.
\end{rem}

\begin{rem}\label{rem:Jacobi_theorem}
A sketch of an alternative proof of Theorem~\ref{thm:main} is as follows.
Consider the following mutually reciprocal formal series
in non-positive powers of $t$:
\[
f(t)=E(-1/t)=\sum_{k=0}^\infty (-1)^k e_k t^{-k},\qquad
g(t)=H(1/t)=\sum_{k=0}^\infty h_k t^{-k}.
\]
Similarly to Lemma~\ref{lem:from_minor_Ta_to_minor_Tb},
\[
\det T_n(a)_{\rho,\si}
=a_p^m \det T_{n+p}(f)_{\widehat{\rho},\widehat{\si}}.
\]
The upper triangular Toeplitz matrices generated by $f$ and $g$
are mutually inverse.
Thus, by Jacobi's theorem about the complementary minor,
every minor of $T_{n+p}(f)$ can be expressed through
a certain minor of $T_{n+p}(g)$.
In our case,
\[
\det T_{n+p}(f)_{\widehat{\rho},\widehat{\si}}
=(-1)^{|\widehat{\xi}|+|\widehat{\eta}|}
\det T_{n+p}(g)_{\widehat{\eta},\widehat{\xi}}
=(-1)^{|\xi|+|\eta|+mp}
\det \bigl[ h_{\widehat{\xi}_k-\widehat{\eta_j}} \bigr]_{j,k=1}^{d+p}.
\]
In order to apply the Jacobi--Trudi formula
and to identify the obtained minor of $T_{n+p}(g)$ with $s_{\al/\be}$,
we have to represent $\widehat{\xi}_k-\widehat{\eta}_j$ as $\al_j-\be_k-j+k$.
This is achieved by taking $\al_j=m+j-\widehat{\eta}_j$
and $\be_k=m+k-\widehat{\xi}_k$.
The corresponding partitions are $\al$ and $\be$ from \eqref{eq:albe_through_xieta}:
\begin{align*}
\al &= m^{p+d}+\id_{d+p}-\widehat{\eta}=(m^p,m^d+\id_d-\eta),\\
\be &= m^{p+d}+\id_{d+p}-\widehat{\xi}=(m^d+\id_d-\xi,0^p).
\end{align*}
The importance of Jacobi's theorem for the theory of Toeplitz matrices
was already noted in various papers, including \cite{BW2006}.
\end{rem}

Let us illustrate Theorem~\ref{thm:main} by an example.


\begin{exmpl}
For $w=5$ and $p=2$, the Laurent polynomial
\eqref{eq:Laurent} takes the form
$a(t) = \sum_{k=-3}^2 a_k t^k$.
Let $a_2=1$, $n=7$, $\xi=(3,6)$ and $\eta=(3,7)$.
Then $d=2$, $m=5$,
$\rho=(1,2,4,5,7)$ and $\si=(1,2,4,5,6)$.
By striking out the rows $\xi$ and the columns $\eta$
in the Toeplitz matrix $T_7(a)$ we obtain the minor
\[
\det(T_7(a)_{\rho,\si}) = 
\detmatr{ccccc}{
a_0 & a_{-1} & a_{-3} & 0 & 0 \\
a_1 & a_0 & a_{-2} & a_{-3} & 0 \\
0 & a_2 & a_0 & a_{-1} & a_{-2} \\
0 & 0 & a_1 & a_0 & a_{-1} \\
0 & 0 & 0 & a_2 & a_1
}.
\]
By Vieta's formulas the $a_k$ are expressed through $e_{2-k}$:
\[
\det(T_7(a)_{\rho,\si}) = 
\detmatr{ccccc}{
e_2 & -e_3 & -e_5 & 0 & 0 \\
-e_1 & e_2 & e_4 & -e_5 & 0 \\
0 & e_0 & e_2 & -e_3 & e_4 \\
0 & 0 & -e_1 & e_2 & -e_3 \\
0 & 0 & 0 & e_0 & -e_1
}
=
\detmatr{ccccc}{
e_2 & -e_3 & -e_5 & e_6 & -e_7 \\
-e_1 & e_2 & e_4 & -e_5 & e_6 \\
-e_{-1} & e_0 & e_2 & -e_3 & e_4 \\
e_{-2} & -e_{-1} & -e_1 & e_2 & -e_3 \\
e_{-4} & -e_{-3} & -e_{-1} & e_0 & -e_1
}.
\]
Using~\eqref{eq:lamu_through_xieta},
we have $\la=(5^2,6-2,3-1)=(5^2,4,2)$ and $\mu=(7-2,3-1)=(5,2)$, then
\[
\det(T_n(a)_{\rho,\si})
= (-1)^{10+19+18} s_{(5,5,4,2)/(5,2)}
= - s_{(5,4,2)/(2)}.
\]
By~\eqref{eq:albe_through_xieta},
$\al=(5,5,3)$ and $\be=(3,1)$, thus
\[
\det(T_n(a)_{\rho,\si}) = - s_{(5,5,3)/(3,1)}.
\]
Note that $\al/\be$ is related with $\la/\mu$ by the flip operation $\ast$:
\[
\la/\mu
= \ydiagram[*(white)] {5+0,2+3,4,2} *[*(black)]{5,5}
= \ydiagram{2+3,4,2}\;,\qquad
\al/\be = \ydiagram{3+2,1+4,3}\;.
\]
\end{exmpl}

\begin{rem}\label{rem:Baxter_Schmidt_and_Trench_formulas}
The Schur function $s_{(n^p)}$ from the right-hand side of \eqref{eq:Tdet}
can be written by the Jacobi--Trudi formula \eqref{SS_JT1}.
Then~\eqref{eq:Tdet} takes the form
\begin{equation}\label{eq:Baxter_Schmidt}
\det T_n(a)=a_p^n (-1)^{pn}
\det \bigl[ h_{n-j+k} \bigr]_{j,k=1}^{p}.
\end{equation}
Baxter and Schmidt \cite{BaxterSchmidt1961}
proved \eqref{eq:Baxter_Schmidt} without using the language of Schur functions.
They denoted by $h_j$ the coefficients of the reciprocal series
to a given series $\sum_k (-1)^k e_k t^k$.
Essentially, they proved the Jacobi theorem about the complementary minor
for this particular situation.
In the banded case, the Schur polynomial $s_{(n^p)}$ can also be written
by \eqref{eq:Schur_quotient}, as the quotient of two determinants:
\begin{equation}\label{eq:Trench_det}
\det T_n(a)=a_p^n (-1)^{pn}
\frac{A_{(n^p)}(x_1,\ldots,x_w)}{V(x_1,\ldots,x_w)}.
\end{equation}
In this form \eqref{eq:Tdet} was deduced by Trench~\cite[Theorem~2]{Trench1985eig}.
These classical results are also explained in \cite[Chapter~2]{BG2005}.
\end{rem}

\begin{prop}\label{prop:skew_Schur_polynomial_as_Toeplitz_minor}
Let $\la/\mu$ be a skew partition and $d=\ell(\la)$.
Define $a$ by \eqref{eq:Laurent} with $p=0$ and $a_p=1$.
Furthermore, let $n=\la_1+d$,
\begin{equation}\label{eq:xieta_from_lamu}
\begin{aligned}
\xi&=\reverse{\la}+\id_d
=(\la_d+1,\la_{d-1}+2,\ldots,\la_1+d),\\
\eta&=\reverse{\mu}+\id_d
=(\mu_d+1,\mu_{d-1}+2,\ldots,\mu_1+d),
\end{aligned}
\end{equation}
and let $\rho,\si\in I^d_n$ be the complements of $\xi,\eta$,
respectively.
Then
\begin{equation}\label{eq:skew_Schur_through_Toeplitz_minor}
s_{\la/\mu}(z_1,\ldots,z_w)
=(-1)^{\total{\la}+\total{\mu}} \det T_n(a)_{\rho,\sigma}.
\end{equation}
\end{prop}

\begin{proof}
Formulas \eqref{eq:xieta_from_lamu} and \eqref{eq:skew_Schur_through_Toeplitz_minor}
are equivalent to \eqref{eq:lamu_through_xieta}
and \eqref{eq:main}, if $a_p=1$ and $p=0$.
\end{proof}

\begin{rem}
Given a skew partition $\la/\mu$ and a number $w$,
the skew Schur polynomial $s_{\la/\mu}$
can be obtained not only as Proposition~\ref{prop:skew_Schur_polynomial_as_Toeplitz_minor} states.
For example, one can choose an arbitrary $n$ satisfying $n\ge\la_1+d$.
Furthermore, instead of \eqref{eq:lamu_through_xieta}
one could use \eqref{eq:albe_through_xieta} with $p=0$
and define $\xi$ and $\eta$ by
\begin{equation}\label{eq:xieta_from_Alexandersson}
\xi=(\mu_1-(m+1),\ldots,\mu_d-(m+d)),\quad
\eta=(\la_1-(m+1),\ldots,\la_d-(m+d)).
\end{equation}
We do not have a general necessary and sufficient condition to determine whether two banded Toeplitz minors correspond to the same skew Schur polynomial.
Two different skew partitions can induce the same skew Schur polynomial,
and the problem of coincidences between skew Schur polynomials is not trivial,
see \cite{ReinerShawWilligenburg2007}.
\end{rem}

\begin{rem}\label{rem:test}
We tested Theorems~\ref{thm:main}, \ref{thm:adj}, \ref{thm:eigvec}
using symbolic computations with Schur functions in SageMath~\cite{SageMath,SageCombinat},
which integrates the Littlewood--Richardson Calculator library (lrcalc)
developed by Anders S. Buch.
In particular, we verified \eqref{eq:main} and \eqref{eq:Alexandersson}
for all possible $p,n,m,\rho,\si$ with $0\le p,n\le 8$ and $a_p=1$.
There were 138812 nonzero answers from 158193 minors in total.
In each example, the determinant in the left-hand side of \eqref{eq:main}
was computed by the recursive expansion.
Such symbolic computations can take much time even for modest values of $m$, say for $m=10$.
We also tested Theorem~\ref{thm:main} 
with pseudorandom values of $p,d,w,n$ ($p,d,w\le 30$, $n\le 200$),
pseudorandom subsets $\xi,\eta$ of $\{1,\ldots,n\}$
and pseudorandom numbers $z_1,\ldots,z_w$.
\end{rem}

\section{Cofactors of banded Toeplitz matrices}
\label{sec:adj}

The purpose of this section is to prove Theorem~\ref{thm:adj}.
Before a formal proof, we show in the next example
how to expand skew Schur functions of the form $s_{(n^p,s)/(r)}$.

\begin{exmpl} \label{exmpl:special_expansion}
The skew Pieri rule \eqref{eq:skew_Pieri} applied to
$(8^3,5)/(2)$ yields
\[
\ydiagram[*(white)] {2+6,8,8,5} *[*(black)]{8,8,8,5}
\;=\;\ydiagram[*(white)] {8,8,6,5} *[*(black)]{8,8,8,5}
\;+\;\ydiagram[*(white)] {8,8,7,4} *[*(black)]{8,8,8,5}
\;+\;\ydiagram[*(white)] {8,8,8,3} *[*(black)]{8,8,8,5}\;,
\]
i.e.
\[
s_{(8^3,5)/(2)}=s_{(8^2,6,5)}+s_{(8^2,7,4)}+s_{(8^2,8,3)}.
\]
The form of the initial diagram implies
that the deleted nodes appear only in the last two rows
of the diagrams in the right-hand side.

The same skew Schur function can be represented as the determinant
of the following matrix, by Jacobi--Trudi formula:
\[
JT((8^3,5)/(2))=
\matr{cccc}{
h_{6} & h_{9} & h_{10} & h_{11} \\
h_{5} & h_{8} & h_{9} & h_{10} \\
h_{4} & h_{7} & h_{8} & h_{9} \\
h_{0} & h_{3} & h_{4} & h_{5}
}.
\]
The columns $2,3,4$ of $JT((8^3,5)/(2))$ coincide
with the columns $2,3,4$ of $JT((8^3,5))$.
In each row of this submatrix,
the degrees of the complete homogeneous polynomials form arithmetic progressions.
The first column of $JT((7^3,5)/(2))$ disturbs this simple structure.
Therefore, we expand the determinant along the first column
and convert each cofactor into a Schur function 
(it is easily done with the rule from Remark~\ref{rule:partition_subindex_h}):
\[
s_{(8^3,5)/(2)} = 
h_{6} s_{(8^2,5)} - h_{5} s_{(9,8,5)} + h_{4} s_{(9^2,5)} - h_{0} s_{(9^3)}.
\]
\end{exmpl}

\begin{lem}\label{lem:LR_specialcase}
Let $n,p\in\bN$ and $r,s\in\{1,\ldots,n\}$, then
\begin{equation}\label{eq:LR_specialcase}
s_{(n^p,s)/(r)} = \sum_{k=\max(0,s-r)}^{\min(n-r,s)}
s_{(n^{p-1},n+s-r-k,k)}.
\end{equation}
\end{lem}

\begin{proof}
We apply the skew Pieri rule \eqref{eq:skew_Pieri},
like in Example~\ref{exmpl:special_expansion}.
For $\la=(n^p,s)$, conditions \eqref{eq:formal_skew_Pieri} become
\[
n\le \nu_j\le n\quad (1\le j\le p-1),\quad
s\le \nu_p\le n,\quad 0\le \nu_{p+1}\le s,\quad
np+s-|\nu|=r.
\]
Consequently, $\nu_1=\ldots=\nu_{p-1}=n$,
and if $\nu_{p+1}$ is denoted by $k$, then $\nu_p=n+s-r-k$.
The corresponding diagram looks as follows
(the deleted nodes are filled with the gray background):
\[
\nu=
\begin{array}{|K{10ex}|K{10ex}|K{10ex}|K{10ex}|}\hline
\multicolumn{4}{|K{40ex}|}{ \bigstrut n^{p-1} } \\\hline
\multicolumn{3}{|K{30ex}|}{ n+s-r-k } & \cellcolor{black!20} r+k-s \\\hline
k &\cellcolor{black!20} s-k \\\cline{1-2}
\end{array}\;.
\]
The conditions $0\le k\le s$ and $s\le n+s-r-k\le n$
determine the limits of the summation in \eqref{eq:LR_specialcase}.
\end{proof}

\begin{proof}[Proof of Theorem~\ref{thm:adj}.]
Applying \eqref{eq:main} and \eqref{eq:Alexandersson}
with $d=1$, $\xi=(s)$ and $\eta=(r)$, we immediately obtain
\eqref{eq:adj_through_skew_Schur}
and \eqref{eq:adj_through_skew_Schur_Alexandersson}.
Combining Lemma~\ref{lem:LR_specialcase} with formula \eqref{eq:adj_through_skew_Schur},
we get \eqref{eq:adj_through_Schur}.
In order to prove the identity \eqref{eq:Trench},
consider the skew Schur function from \eqref{eq:adj_through_skew_Schur_Alexandersson}:
\[
s_{((n-1)^p,n-r)/(n-s)} = 
\left|\begin{matrix}
h_{s-1} & h_{n} & \cdots & h_{n+p-1} \\
\vdots & \vdots & \ddots & \vdots \\
h_{s-p} & h_{n-p+1} &\cdots & h_{n} \\
h_{s-r-p} & h_{n-r-p+1} & \cdots & h_{n-r}
\end{matrix}\right|.
\]
As in the Example~\ref{exmpl:special_expansion}, we expand the determinant
by the first column and use Remark~\ref{rule:partition_subindex_h}:
\[
s_{((n-1)^p,n-r)/(n-s)} 
= \sum_{j=1}^p (-1)^j h_{s-j} s_{(n^{j-1},(n-1)^{p-j},n-r)}
+ (-1)^p h_{s-r-p} s_{(n^p)}.
\]
Making the change of variable $k=j-p$ we obtain
\[
s_{((n-1)^p,n-r)/(n-s)} 
= (-1)^p h_{s-r-p} s_{(n^p)} - (-1)^p
\sum_{k=0}^{p-1} (-1)^k h_{s+k-p} s_{(n^{p-k-1},(n-1)^k,n-r)}.
\]
The substitution of the latter expression into \eqref{eq:adj_through_skew_Schur}
leads to \eqref{eq:Trench}.
\end{proof}

\begin{rem}
Identity \eqref{eq:Trench} can also be verified by applying Pieri's formula.
The sum in the right-hand side of \eqref{eq:Trench}
can be expanded into a certain telescopic sum in terms of Schur functions, but the corresponding formulas are rather complicated, and we decided to omit that proof.
\end{rem}

\begin{rem}
The number of summands in the right-hand side of \eqref{eq:adj_through_Schur}
does not depend on $p$ and can be described
as the distance from the cell $(r,s)$ to the boundary of the matrix:
\begin{equation}\label{eq:number_of_summands}
\min\{n-r,s+1\}-\max\{0,s-r\}+1=\min\{r,s,n+1-r,n+1-s\}.
\end{equation}
For example, if $n=5$, then \eqref{eq:number_of_summands} yields the following table:
\[
\matr{K{2ex}K{2ex}K{2ex}K{2ex}K{2ex}}{
1 & 1 & 1 & 1 & 1 \\
1 & 2 & 2 & 2 & 1 \\
1 & 2 & 3 & 2 & 1 \\
1 & 2 & 2 & 2 & 1 \\
1 & 1 & 1 & 1 & 1
}.
\]
Plotting the 3D graph of the function $(r,s)\mapsto \eqref{eq:number_of_summands}$,
i.e. representing the number of summands by the height, we obtain a pyramid.
\end{rem}

We already know that all entries of $\adj(T_n(a))$
can be written as skew Schur functions, up to certain coefficients.
In the case of the entries belonging to the ``border'' of the matrix $\adj(T_n(a))$,
these skew Schur functions simplify to Schur functions.
The next result, being a corollary of Theorem~\ref{thm:adj},
gives a formula for the entries of the first column.

\begin{cor}\label{cor:adj_first_column}
Let $a$ be of the form \eqref{eq:aseries_through_e} and $n\ge 1$.
Then for every $r$ in $\{1,\ldots,n\}$
\begin{equation}\label{eq:adj_first_column}
\adj(T_n(a))_{r,1}=a_p^{n-1} (-1)^{p(n-1)} s_{((n-1)^{p-1},n-r)}.
\end{equation}
\end{cor}

\begin{proof}
Apply \eqref{eq:adj_through_skew_Schur_Alexandersson} with $s=1$
and the identity $s_{((n-1)^p,n-r)/(n-1)}=s_{((n-1)^{p-1},n-r)}$.
Alternatively, use \eqref{eq:adj_through_skew_Schur} with $s=1$
and the flip operation ${}^\ast$.
\end{proof}

\section{Eigenvectors of banded Toeplitz matrices}
\label{sec:eigvec}

The starting point of our approach to eigenvectors
is the elementary observation that for a non-invertible square matrix $A$,
\[
A\,\adj(A) = \det(A) I_n = 0_{n\times n}.
\]
Thus every column of the adjugate matrix $\adj(A)$
belongs to the null-space of $A$.
Applying this reasoning to $A-xI$ instead of $A$
we arrive at the following statement
which was already mentioned and used in \cite[Section~2]{BGM2010}.

\begin{lem}\label{lem:coladj}
Let $A\in\bC^{n\times n}$ and $x$ be an eigenvalue of $A$.
Then every column $v$ of the matrix $\adj(A-xI_n)$ satisfies $Av=xv$.
\end{lem}

\begin{proof}[Proof of Theorem~\ref{thm:eigvec}.]
Let us define $v\in\bC^n$ as the first column of $\adj(T_n(a-x))$ multiplied by $(-1)^{p(n-1)}$.
Then $v_r=s_{((n-1)^{p-1},n-r)}$ by \eqref{eq:adj_first_column}, and by Lemma~\ref{lem:coladj} $T_n(a)=xv$.
This part is not only true for Laurent polynomials,
but also for Laurent series of the form \eqref{eq:aseries_through_e}.


Now suppose that $a$ is a Laurent polynomial,
and $a-x$ has $w$ distinct zeros $z_1,\ldots,z_w$.
Recall the notation~\eqref{eq:generalized_Vandermonde}
and consider the matrices
\[
B = A_{((n-1)^{p-1},n-r)}(z_1,\ldots,z_w),\qquad
D = A_{(n^p)}(z_1,\ldots,z_w),
\]
i.e.
\[
B =
\matr{cK{7ex}c}{
\medstrut z_1^{n+w-2} & \dots & z_w^{n+w-2} \\
\medstrut \vdots & \ddots & \vdots\\
\medstrut z_1^{n+w-p} & \cdots & z_w^{n+w-p} \\
\medstrut \cellcolor{black!15}
z_1^{n-r+w-p}
& \cellcolor{black!15} \dots
& \cellcolor{black!15} z_w^{n-r+w-p} \\
\medstrut z_1^{w-p-1} & \dots & z_w^{w-p-1} \\
\medstrut \vdots & \ddots & \vdots\\
\medstrut z_1^0  & \dots & z_w^0
},
\qquad
D =
\matr{cK{9ex}c}{
\medstrut \cellcolor{black!15} z_1^{n+w-1} &
\cellcolor{black!15} \dots & 
\cellcolor{black!15} z_w^{n+w-1} \\
\medstrut z_1^{n+w-2} & \dots & z_w^{n+w-2} \\
\medstrut \vdots & \ddots & \vdots\\
\medstrut z_1^{n+w-p} & \cdots & z_w^{n+w-p} \\
\medstrut z_1^{w-p-1} & \dots & z_w^{w-p-1} \\
\medstrut \vdots & \ddots & \vdots\\
\medstrut z_1^0  & \dots & z_w^0
}.
\]
After deleting the $p$-th row of $B$
and the first row of $D$
(these rows are filled with gray background),
one obtains the same submatrix.
Therefore the cofactor of the entry $(p,j)$ in $B$
coincides with the cofactor of the entry $(1,j)$ in $D$,
up to the factor $(-1)^{p-1}$.
Denote by $C_j$ this cofactor divided by the
Vandermonde polynomial $V=V(z_1,\ldots,z_w)$:
\begin{equation}\label{eq:C_through_adj}
C_j
= \frac{(\adj B)_{j,p}}{V}
= \frac{(-1)^{p-1}(\adj D)_{j,1}}{V}.
\end{equation}
The matrix $D$ does not depend on $r$,
hence the numbers $C_j$ neither depend on $r$.
Here is a more explicit formula for $C_j$:
\[
C_j=\frac{(-1)^{p+j}}{V}
\detmatr{cccccc}{
\medstrut z_1^{n+w-2} & \dots & z_{r-1}^{n+w-2} &
z_{r+1}^{n+w-2} & \dots & z_w^{n+w-2} \\
\medstrut \vdots & \ddots & \quad &
\vdots & \ddots & \vdots \\
\medstrut z_1^{n+w-p} & \cdots & z_{r-1}^{n+w-p} &
z_{r+1}^{n+w-p} & \cdots & z_w^{n+w-p} \\
\medstrut z_1^{w-p-1} & \dots & z_{r-1}^{w-p-1} &
z_{r+1}^{w-p-1} & \dots & z_w^{w-p-1} \\
\medstrut \vdots & \ddots & \vdots &
\vdots & \ddots & \vdots \\
\medstrut z_1^0  & \dots & z_{r-1}^0 &
z_{r+1}^0  & \dots & z_w^0
}.
\]
Applying \eqref{eq:Schur_quotient} we represent the Schur polynomial
$s_{((n-1)^{p-1},n-r)}$
from the right--hand side of \eqref{eq:eigvec_components}
as a quotient of two determinants.
Expanding the numerator by the $p$th row we obtain
\begin{equation}\label{eq:v_by_expansion}
v_r = \frac{\det B}{V}=\sum_{j=1}^w C_j B_{p,j},
\end{equation}
which coincides with \eqref{eq:eigvec_Trench}.
\end{proof}

\begin{rem}\label{rem:eigvec_confluent}
Formula \eqref{eq:eigvec_Trench} can be generalized to the case
when some of the numbers $z_1,\ldots,z_w$ coincide.
In this case we apply \eqref{eq:C_through_adj}
and \eqref{eq:v_by_expansion}
with $B,D,V$ defined as in Remark~\ref{rem:confluent}.
\end{rem}

\begin{rem}\label{rem:Trench_eigvec}
The vector $C=[C_j]_{j=1}^w$, multiplied by $(-1)^{p-1}$,
is the $p$-th column of the matrix $\adj(D)$.
Since $\det(D)=0$, this vector belongs to the null-space of $D$.
In \cite{Trench1985eig}, Trench wrote Toeplitz matrices
in the transposed form $[a_{k-j}]_{j,k=1}^n$,
thus instead of \eqref{eq:eigvec_Trench}
he obtained a linear combination
of geometric progressions with \emph{increasing} powers.
Up to these technical changes, he described the vector $C$
as a nonzero solution of the linear system $DC=0_w$.
Thus the result obtained by Trench is slightly more general
than \eqref{eq:eigvec_Trench}.
\end{rem}

\begin{exmpl}\label{exmpl:eigvec_Hessenberg_Toeplitz}
If $a$ is of the form \eqref{eq:aseries_through_e} with $p=1$,
then the corresponding Toeplitz matrices $T_n(a)$
have only one nonzero diagonal below the main diagonal
and are known as \emph{upper Hessenberg--Toeplitz matrices}.
In this case, \eqref{eq:eigvec_components} simplifies to
\begin{equation}\label{eq:eigvec_Hessenberg_Toeplitz}
v_r=h_{n-r}.
\end{equation}
An analogue of the formula \eqref{eq:eigvec_Hessenberg_Toeplitz}
for lower Hessenberg--Toeplitz matrices was established
in \cite[Theorem 1.1]{BBGM2012}.
\end{exmpl}

\section*{Acknowledgements}

E.~A.~Maximenko is grateful to Sergei M.\ Grudsky, Albrecht B\"{o}ttcher,
and Johan Manuel Bogoya for joint investigations on Toeplitz matrices,
to Rom\'{a}n Higuera Garc\'{i}a for the joint study
of Schur polynomials and Toeplitz eigenvectors,
and to Gabino S\'{a}nchez Arzate for the joint study
of Trench's article \cite{Trench1985inv}.

\bigskip\noindent
Egor A. Maximenko\\
\myurl{http://orcid.org/0000-0002-1497-4338}\\
e-mail: emaximenko@ipn.mx

\medskip\noindent
Mario Alberto Moctezuma-Salazar\\
\myurl{http://orcid.org/0000-0003-4579-0598}\\
e-mail: m.a.mocte@gmail.com

\bigskip\noindent
Instituto Polit\'{e}cnico Nacional\\
Escuela Superior de F\'{i}sica y Matem\'{a}ticas\\
Apartado Postal 07730\\
Ciudad de M\'{e}xico, Mexico


\begin{thebibliography}{15}

\bibitem{A2012}
Alexandersson, P.\ (2012):
Schur polynomials, banded Toeplitz matrices and Widom's formula,
Electron.\ J.\ Combin.\ 19:4, P22,\\
\myurl{http://www.combinatorics.org/ojs/index.php/eljc/article/view/v19i4p22}.

\bibitem{BG2017}
Barrera, M.; Grudsky, S.M.\ (2017):
Asymptotics of eigenvalues for pentadiagonal symmetric Toeplitz matrices,
Oper. Theory Adv. Appl.~259, 51--77,\\
\doi{10.1007/978-3-319-49182-0\_7}.

\bibitem{BaxterSchmidt1961}
Baxter, G.; Schmidt, P.\ (1961):
Determinants of a certain class of non-Hermitian Toeplitz matrices,
Math.\ Scand.\ 9, 122--128,\\
\myurl{http://www.mscand.dk/article/view/10630/8651}.

\bibitem{BBGM2012}
Bogoya, J.M.; B\"{o}ttcher, A.; Grudsky, S.M.; Maksimenko, E.A.\ (2012):
Eigenvectors of Hessenberg Toeplitz matrices and a problem by Dai, Geary, and Kadanoff,
Linear Algebra Appl.\ 436:9, 3480--3492,
\doi{10.1016/j.laa.2011.12.012}.

\bibitem{BBGM2015}
Bogoya, J.M.; B\"{o}ttcher, A.; Grudsky, S.M.; Maximenko, E.A.\ (2015):
Maximum norm versions of the Szeg\H{o} and Avram-Parter theorems for Toeplitz matrices,
J.\ Approx.\ Theory 196, 79--100,
\doi{10.1016/j.jat.2015.03.003}.

\bibitem{BBGM2016}
Bogoya, J.M.; B\"{o}ttcher, A.; Grudsky, S.M.; Maximenko, E.A.\ (2016):
Eigenvectors of Hermitian Toeplitz matrices with smooth simple-loop symbols,
Linear Algebra Appl.\ 493, 606--637,
\doi{10.1016/j.laa.2015.12.017}.

\bibitem{BorodinOkounkov2000}
Borodin, A.; Okounkov, A.\ (2000):
A Fredholm determinant formula for Toeplitz determinants,
Integr.\ Equ.\ Oper.\ Theory 37:4, 386--396,
\doi{10.1007/BF01192827}.

\bibitem{BG2005}
B\"{o}ttcher, A.; Grudsky, S.M.\ (2005):
Spectral Properties of Banded Toeplitz Matrices, SIAM, Philadelphia,
\doi{10.1137/1.9780898717853}.

\bibitem{BGM2010}
B\"{o}ttcher, A.; Grudsky, S.M.; Maksimenko, E.A.\ (2010):
On the structure of the eigenvectors of large Hermitian Toeplitz band matrices,
Oper. Theory Adv. Appl.\ 210, 15--36,
\doi{10.1007/978-3-0346-0548-9\_2}.

\bibitem{BFGM2015}
B\"{o}ttcher, A.; Fukshansky, L.; Garcia, S.R.; Maharaj, H.\ (2015):
Toeplitz determinants with perturbations in the corners,
J.\ Funct.\ Analysis 268:1, 171--193,
\doi{10.1016/j.jfa.2014.10.023}.

\bibitem{BS1999}
B\"{o}ttcher, A.; Silbermann, B.\ (1999):
Introduction to Large Truncated Toeplitz Matrices,
Springer-Verlag, New York.

\bibitem{BS2006}
B\"{o}ttcher, A.; Silbermann, B.\ (2006):
Analysis of Toeplitz Operators, 2nd ed.,
Springer-Verlag, Berlin, Heidelberg, New York,
\doi{10.1007/3-540-32436-4}.

\bibitem{BW2006}
B\"{o}ttcher, A.; Widom, H.\ (2006):
Szeg\"{o} via Jacobi,
Linear Algebra Appl.\ 419, 656--667,
\doi{10.1016/j.laa.2006.06.009}.

\bibitem{BumpDiaconis2002}
Bump, D.; Diaconis, P.\ (2002):
Toeplitz minors,
J.\ Combin.\ Theory Ser.\ A 97, 252--271,
\doi{10.1006/jcta.2001.3214}.

\bibitem{DeiftItsKrasovsky2013}
Deift, P.; Its, A.; Krasovsky, I.\ (2013):
Toeplitz matrices and Toeplitz determinants under the impetus of the Ising model.
Some history and some recent results,
Comm.\ Pure Appl.\ Math. 66:9, 1360--1438,
\doi{10.1002/cpa.21467}.

\bibitem{Elouafi2014}
Elouafi, M.\ (2014):
On a relationship between Chebyshev polynomials and Toeplitz determinants,
Appl.\ Math.\ Comput.\ 229:1, 27--33,
\doi{10.1016/j.amc.2013.12.029}.

\bibitem{Elouafi2015}
Elouafi, M.\ (2015):
A Widom like formula for some Toeplitz plus Hankel determinants,
J.\ Math.\ Anal.\ Appl.\ 422:1, 240--249,
\doi{10.1016/j.jmaa.2014.08.043}.

\bibitem{GaroniSerra2016}
Garoni, C.; Serra-Capizzano, S.\ (2016):
The theory of locally Toeplitz sequences: a review, an extension, and a few representative applications,
Bol.\ Soc.\ Mat.\ Mex.\ 22:2, 529--565,
\doi{10.1007/s40590-016-0088-8}.

\bibitem{Gessel1990}
Gessel, I.M.\ (1990):
Symmetric functions and P-recursiveness,
J.\ Combin.\ Theory Ser.\ A 53:2, 257--285,
\doi{10.1016/0097-3165(90)90060-A}.

\bibitem{GrenanderSzego1958}
Grenander, U.; Szeg\H{o}, G.\ (1958):
Toeplitz Forms and Their Applications,
University of California Press, Berkeley, Los Angeles.

\bibitem{Lascoux2003}
Lascoux, A.\ (2003):
Symmetric Functions and Combinatorial Operators on Polynomials,
In series: CBMS Regional Conference Series in Mathematics, vol.~99,
co-publication of the AMS and Conference Board of the Mathematical Sciences,
Providence, Rhode Island,
\doi{10.1090/cbms/099}.

\bibitem{Littlewood1950}
Littlewood, D.E.\ (1950):
The Theory of Group Characters and Matrix Representations of Groups,
2nd ed., The Clarendon Press, Oxford University Press, Oxford.

\bibitem{Macdonald1995}
Macdonald, I.G.\ (1995):
Symmetric functions and Hall polynomials, 2nd ed.,
The Clarendon Press, Oxford University Press, Oxford.

\bibitem{ReinerShawWilligenburg2007}
Reiner, V.; Shaw, K.M.; Willigenburg, S.~van\ (2007):
Coincidences among skew Schur functions,
Adv.\ Math.\ 216:1, 118--152,
\doi{10.1016/j.aim.2007.05.006}.

\bibitem{SageMath}
The Sage Developers (2017):
{S}ageMath, the {S}age {M}athematics {S}oftware {S}ystem ({V}ersion 7.5.1),
\myurl{http://www.sagemath.org},
\doi{10.5281/zenodo.28514}.

\bibitem{SageCombinat}
The {S}age-{C}ombinat community (2008):
{S}age-{C}ombinat: enhancing Sage as a toolbox
for computer exploration in algebraic combinatorics,
\myurl{http://combinat.sagemath.org}.

\bibitem{Stanley1999}
Stanley, R.P.\ (1999):
Enumerative Combinatorics, vol.\ 2,
Cambridge University Press, Cambridge.

\bibitem{TracyWidom2001}
Tracy, C.A.; Widom, H.\ (2001):
On the distributions of the lengths of the longest monotone subsequences in random words,
Probab.\ Theory Relat.\ Fields 119, 350--380,
\doi{10.1007/s004400000107}.

\bibitem{Trench1985inv}
Trench, W.F.\ (1985):
Explicit inversion formulas for Toeplitz band matrices,
SIAM\ J.\ on Algebraic and Discrete Methods
(transformed to SIAM J.\ Matrix Anal.\ Appl.) 6:4, 546--554,
\doi{10.1137/0606054}.

\bibitem{Trench1985eig}
Trench, W.F.\ (1985):
On the eigenvalue problem for Toeplitz band matrices,
Linear Algebra Appl.\ 64, 199--214,
\doi{10.1016/0024-3795(85)90277-0}.

\bibitem{Widom1958}
Widom, H.\ (1958):
On the eigenvalues of certain Hermitean operators,
Trans.\ Amer.\ Math.\ Soc.\ 88:2, 491--522,
\doi{10.2307/1993228}.

\end{thebibliography}
\end{document}